\documentclass[preprint,12pt]{elsarticle}
\usepackage{amsmath,amssymb,amsthm}
\usepackage{xcolor}
\usepackage{graphicx}
\usepackage{geometry}
\usepackage{caption}
\usepackage{hyperref}
\usepackage{float}
\usepackage{booktabs}
\usepackage{array}

\newtheorem{theorem}{Theorem}[section]
\newtheorem{lemma}[theorem]{Lemma}

\newtheorem{remark}{Remark}

\DeclareMathOperator*{\esup}{ess\ sup}

\journal{...}

\begin{document}

\begin{frontmatter}

\title{Stochastic representation of solutions for the parabolic Cauchy problem with variable exponent coefficients}
\author{Mustafa Avci}

\affiliation{organization={Faculty of Science and Technology, Applied Mathematics\\  Athabasca University},
            city={Athabasca},
            postcode={T9S 3A3},
            state={AB},
            country={Canada}}
\ead{mavci@athabascau.ca (primary) & avcixmustafa@gmail.com}

\begin{abstract}
In this work, we prove existence and uniqueness of a bounded viscosity solution for the Cauchy problem of degenerate parabolic equations with variable exponent coefficients. We construct the solution directly using the stochastic representation, then verify it satisfies the Cauchy problem. The corresponding SDE, on the other hand, allows the drift and diffusion coefficients to respond nonlinearly to the current state through the state-dependent variable exponents, and thus, extends the expressive power of classical SDEs to better capture complex dynamics. To validate our theoretical framework, we conduct comprehensive numerical experiments comparing finite difference solutions (Crank-Nicolson on logarithmic grids) with Monte Carlo simulations of the SDE.
\end{abstract}


\end{frontmatter}

\section{Introduction}
We study a new class of degenerate parabolic partial differential equations (PDEs) with variable exponent $p(x), q(x)$ coefficients of the form
\begin{equation}\label{eq.1}
\begin{cases}
\begin{array}{rll}
\partial_{t}u& =\frac{1}{2}\sigma^2 x^{2q(x)}\partial^{2}_{x}u+\mu x^{p(x)}\partial_{x}u-V(x)u,\quad (x,t) \in (0,\infty) \times (0,T)  \\
u(x,0)&=f(x), \quad x \in (0,\infty) \tag{{${\mathcal{P}}$}}
\end{array}
\end{cases}
\end{equation}
and establish existence and uniqueness of bounded viscosity solutions, where $V \in C((0,\infty))$ is the potential function with $V(x)\geq 0$; $f \in C_b^{2}((0,\infty))$; $p(\cdot), q(\cdot): (0, \infty) \to (0,\infty)$ are differentiable, variable exponent functions; and $\mu, \sigma>0$ are constants.\\
We concurrently prove that the associated stochastic differential equation (SDE) with state-dependent variable exponent functions $p(X(t)), q(X(t))$
\begin{equation}\label{eq.1a}
\begin{cases}
\begin{array}{rll}
dX(t) &= \mu X(t)^{p(X(t))}dt + \sigma X(t)^{q(X(t))}dW(t),\,\, t\in [0,T], \\
X_{0}&=x, \quad x \in (0,\infty),\tag{{${\mathcal{GM}}$}}
\end{array}
\end{cases}
\end{equation}
admits a unique strong solution $X(t)$ that remains strictly positive in the state space $(0,\infty)$. Using the dynamic programming principle (DPP) along with the comparison principles for viscosity solutions, we establish a Feynman-Kac representation formula connecting the PDE solution (of \eqref{eq.1}) to the SDE (of \eqref{eq.1a}) via expectation. \\

Parabolic partial differential equations with variable coefficients arise naturally in numerous applications including financial mathematics, stochastic control, and diffusion processes in heterogeneous media. The classical theory of parabolic equations, initiated by the work of Crandall, Ishii, and Lions \cite{cil92} on viscosity solutions, provides powerful tools for existence, uniqueness, and regularity of solutions even when classical differentiability fails. The connection between parabolic PDE and stochastic processes, formalized through the celebrated Feynman-Kac formula \cite{feynman48, kac49}, offers both theoretical insights and practical computational methods via probabilistic representations.\\

Our novel approach given by the pair \eqref{eq.1}-\eqref{eq.1a}, on the other hand, introduces a fundamentally different structure. Here, the diffusion coefficient is $\sigma^2 x^{2q(x)}$ and the drift coefficient is $\mu x^{p(x)}$. This is qualitatively different from classical parabolic partial differential equations with constant or variable coefficients in several crucial aspects described as follows:
\begin{itemize}
    \item \textbf{Nonlinear principal part}: The exponents $p(x)$ and $q(x)$ directly modulate the strength of diffusion and drift at each spatial location, creating a genuinely nonlinear operator even when the solution $u$ appears linearly.

    \item \textbf{Stochastic interpretation}: Equation \eqref{eq.1} arises naturally as the Kolmogorov backward equation for the stochastic differential equation \eqref{eq.1a}, where the diffusion and drift are state-dependent nonlinear functions of the process itself. Such SDEs model financial assets with state-dependent volatility (e.g., the Constant-Elasticity-of-Variance (CEV) model, a special case of \eqref{eq.1a} for $p(x)=1, q(x)\geq 1$), population dynamics with density-dependent growth rates, and adaptive control systems.

    \item \textbf{Degenerate behavior}: When $q(x) > 1$ near $x=0$, the equation exhibits enhanced diffusion, while $q(x) < 1$ leads to degeneracy. Similarly, $p(x) \neq 1$ creates nonlinear drift. The interplay between these effects requires new analytical techniques beyond classical parabolic theory.

    \item \textbf{Connection to financial mathematics}: The structure in \eqref{eq.1}-\eqref{eq.1a} generalizes the Black-Scholes framework (through the geometric Brownian motion (GBM), a special case of \eqref{eq.1a} for $p(x)=q(x)=1$) to incorporate realistic features such as leverage effects, volatility smiles, and state-dependent risk premia, which are empirically observed in financial markets \cite{ dupire94,heston93}.
\end{itemize}
The remainder of the paper is organized as follows. In Section 2, we present preliminary material on the measure-theoretic probability and the theory of stochastic analysis as well as establish the existence, uniqueness, and positivity of strong solutions to \eqref{eq.1a}. In Section 3, we prove the existence and uniqueness of bounded viscosity solutions to \eqref{eq.1} and establish the Feynman-Kac stochastic representation. In Section 4, we describe the numerical methods and presents comprehensive computational validation of the theoretical results.

\section{Preliminaries and auxiliary results} \label{Sec2}
We start with some basic concepts of the measure-theoretic probability and the theory of stochastic process (see, e.g., \cite{friedman75,karatzasshreve91,mao07,oksendal03}).
Let $W(t)$ be a Brownian motion on a probability space $(\Omega, \mathcal{F}, \mathbb{P})$ under the real-world measure $\mathbb{P}$ and let $\mathcal{F}_t$ be a \emph{filtration} denoted by $\{\mathcal{F}_t\}_{t \geq 0}$. Then $(\Omega, \mathcal{F}, \{\mathcal{F}_t\}_{t \geq 0}, \mathbb{P})$ is called a \emph{filtered probability space}.
A \emph{stochastic process} $X(t)$ is said to be \emph{adapted} to $\mathcal{F}_{t}$ if for each $t \in [t_0,T]$, $X(t)$ is $\mathcal{F}_{t}$-measurable.
Denote by $\mathcal{M}^{m}[t_0,T]$ the class of real-valued $\mathcal{F}_{t}$-adapted stochastic processes $X(t)$  satisfying
\begin{equation*}
\mathbb{E}\int_{t_0}^{T}|X(t)|^{m}dt<\infty \, \text{ if }\, 1\leq m < \infty \quad \text{and} \quad \mathbb{E}\left[\esup_{t_0 \leq t \leq T}|X(t)| \right]<\infty\, \text{ if }\, m=\infty.
\end{equation*}
We define a class of variable exponent functions, denoted by $\mathcal{S}$. We say that the function $h(\cdot)$ belongs to the class $\mathcal{S}$ if it satisfies the following hypotheses:
\begin{itemize}
   \item [($\mathbf{h_1}$)] $h(\cdot): (0, \infty) \to \mathbb{R}$ is a differentiable function satisfying
   \begin{align*}
     1\leq h^-:=\inf_{x> 0} h(x)\quad \text{and} \quad \sup_{x> 0} h(x):=h^+ <\infty.
   \end{align*}
   \item [($\mathbf{h_2}$)]It holds
   \begin{align*}
    \lim_{x \to \infty}h(x)=1 \quad \text{and} \quad \limsup_{x \to \infty}(h(x)-1)\log(x)<\infty.
   \end{align*}
   \item [($\mathbf{h_3}$)] There exist real numbers $\delta, M_0, C_0>0$ and $\alpha>0$ with $h^+<1+\alpha$ such that
   \begin{equation*}
   |h^{\prime}(x)|\leq
   \left\{ \begin{array}{ll}
   M_0,\quad & 0<x\leq \delta, \\
   C_0 x^{-1-\alpha},\quad & x > \delta.
   \end{array}\right.
   \end{equation*}
\end{itemize}
\begin{remark}\label{Rem:1} The function $h(x)=1+0.5e^{-x}$ satisfies the hypotheses $(\mathbf{h_1})$-$(\mathbf{h_3})$. $(\mathbf{h_1})$ and $(\mathbf{h_2})$ easily follows. If one lets $\delta \leq 1$, $M_0>0.5$, $C_0>1$ and $\alpha>0.5$ then $(\mathbf{h_3})$ is also satisfied.
\end{remark}
Consider the one-dimensional SDE
\begin{equation}\label{eq.2}
dX(t) = \mu(X(t),t)dt + \sigma(X(t),t)dW(t), \quad t\in[t_0,T],
\end{equation}
where $\mu:\mathbb{R}\times[t_0,T]\to \mathbb{R}$ and $\sigma:\mathbb{R}\times[t_0,T]\to \mathbb{R}$ are both Borel measurable functions.
Define the conditions $(\mathbf{a_1})-(\mathbf{a_3})$:
\begin{itemize}
  \item [$(\mathbf{a_1})$] There exists a constant $L>0$ such that $\forall x,y\in \mathbb{R}$, $t \in (0,\infty)$
  \begin{equation}\label{eq.2b}
   |\mu(x,t)-\mu(y,t)| + |\sigma(x,t)-\sigma(y,t)| \leq L|x-y|.
  \end{equation}
  \item [$(\mathbf{a_2})$] There exists a constant $K>0$ such that $\forall x\in \mathbb{R}$, $t \in (0,\infty)$
  \begin{equation}\label{eq.2c}
  |\mu(x,t)| + |\sigma(x,t)|\leq K(1+|x|).
  \end{equation}
  \item [$(\mathbf{a_3})$] There exists a constant $\lambda>0$ such that $\forall x\in \mathbb{R}$, $t \in [0,\infty)$
  \begin{equation}\label{eq.2d}
   \sigma^{2}(x,t) \geq \lambda.
\end{equation}
\end{itemize}
Next, we show that \eqref{eq.1a} is well-defined.
\begin{lemma}\label{Lem:2.3}
Assume that $x > 0$, and $p(\cdot), q(\cdot)$ satisfy $(\mathbf{h_1})$, $(\mathbf{h_3})$. Then $X(t)$ is strictly positive a.s. for all $t \in [0,T]$.
\end{lemma}
\begin{proof}
We use the following Feller’s test \cite{elt09}, which is also know as the local (differential) form of Feller’s non-attainability test, to analyze the behavior of the diffusion processes $X(t)$ at the boundary (i.e. $X(t)=0$) of its state space $(0,\infty)$.\\
Consider equation \eqref{eq.2}. If the condition
\begin{equation}\label{eq.4}
\lim_{x \to 0^+}\left(\mu(x)-\frac{1}{2}\partial_{x} \sigma^2(x)\right)\geq 0
\end{equation}
holds, then the boundary $X(t)=0$ is non-attainable for the process $X(t)$ if  $x>0$.\\
Now, we apply \eqref{eq.4} to the diffusion process \eqref{eq.1a}.\\
Then, for $0<x\leq\delta<1$,
\begin{align}\label{eq.7ca}
\mathcal{T}(x)=\mu x^{p(x)}-\frac{\sigma^2}{2}\partial_{x} x^{2q(x)}\geq \mu x^{p(x)}-\sigma^2 \left(M_0x^{2q(x)}|\log(x)|+x^{2q(x)-1}q^+\right).
\end{align}
Using the elementary calculus and the assumptions we have, it follows
\begin{align}\label{eq.8ba}
\lim_{x \to 0^+}\mathcal{T}(x)& \geq \lim_{x \to 0^+} \left(\mu x^{p(x)}-\sigma^2 \left(M_0 x^{2q(x)}|\log(x)|+x^{2q(x)-1}q^+\right)\right)= 0.
\end{align}
Hence $X(t)>0$ a.s. for all $t \in [0,T]$. In particular, the boundary $x=0$ is natural; that is, starting from $x>0$, it is never hit in finite time.
\end{proof}
\begin{lemma}\label{Lem:2.4}
Assume that $p(\cdot), q(\cdot) \in \mathcal{S}$. Assume also that the initial value $X(0)=x_0$ has a finite second moment: $\mathbb{E}[x_0^2]<\infty$, and is independent of $\{W(t),\, t\geq 0\}$. Then there exists a unique strong solution $X(t)$ of \eqref{eq.1a} in $\mathcal{M}^{2}[0,T]$ with continuous
paths. Further,
\begin{equation}\label{eq.8aa}
\mathbb{E}\left[\sup_{0\leq t \leq T}|X(t)|^{2}\right] \leq (1+3\mathbb{E}[x_0^2])e^{6TK^2(\mu^2T+4\sigma^2)}.
\end{equation}
\end{lemma}
\begin{proof}
If one follows similar arguments as with Lemma 3.1 of \cite{avci25}, it is straightforward to show that the drift and the diffusion terms $\mu x^{p(x)}$ and $\sigma x^{q(x)}$ of \eqref{eq.1a} satisfy $(\mathbf{a_1})$ and $(\mathbf{a_2})$. Thus, by Theorem 3.2 of \cite{avci25}, there exists a unique strong solution $X(t)$ of \eqref{eq.1a} in $\mathcal{M}^{2}[0,T]$ with continuous paths satisfying \eqref{eq.8aa}.
\end{proof}
\begin{lemma}\label{Lem:2.1a}
Assume that $X(t)$ solves \eqref{eq.1a}. Then the infinitesimal generator of $X(t)$ is the second order differential operator $\mathcal{L}$ given by
\begin{equation}\label{eq.2a}
(\mathcal{L} g)(x):= \lim_{t \to 0^+}\frac{\mathbb{E}_{x}[g(X(t))]-g(x)}{t}=\frac{1}{2}\sigma^2 x^{2q(x)}\frac{d^2 g(x)}{d x^2}+\mu x^{p(x)}\frac{dg(x)}{d x},\,\, x\in(0,\infty)
\end{equation}
for any $g \in C^{2}_{0}(\mathbb{R})$, where $\mathbb{E}_{x}[\cdot]$ means expectation conditioned on $X_0=x$.
\end{lemma}
\begin{proof}
If we define the process $G(t):=g(X(t))$, and apply It\^{o}'s formula to $G(t)$, we obtain
\begin{align} \label{eq.18f}
dG(t)=d(g(X(t))&=\left(\mu X(t)^{p(X(t))}\frac{dg(X(t))}{d x}+\frac{1}{2}\sigma^2 X(t)^{2q(X(t))}\frac{d^2 g(X(t))}{d x^2} \right)dt\nonumber\\
&+\sigma X(t)^{q(X(t))}\frac{d g(X(t))}{d x} dW(t).
\end{align}
Note that by the assumption $(\mathbf{h_2})$,  $X(t)^{q(X(t))} \in \mathcal{M}^{2}[0,T]$. Indeed:\\
If $0<x\leq 1$, there exists $0<\delta\leq 1$ such that
\begin{equation}\label{eq.18d}
x^{q(x)}\leq \delta^{q^-}\leq \delta^{q^-}(1+x).
\end{equation}
If $1<x<\infty$, there exist real numbers $M_{\infty}, R_{\infty}>0$ such that
\begin{equation}\label{eq.9ab}
q(x)-1 \leq \frac{M_{\infty}}{\log(x)}, \quad \forall x >R_{\infty},
\end{equation}
which yields
\begin{equation}\label{eq.9abc}
e^{(q(x)-1)\log (x)}\leq e^{M_{\infty}} \Rightarrow x^{q(x)}\leq e^{M_{\infty}}(1+x).
\end{equation}
Therefore, for any $x \in (0,\infty)$
\begin{equation}\label{eq.9cbd}
x^{q(x)}\leq K(1+x), \quad \text{where} \quad K:= \max\{\delta^{q^-},\delta^{p^-}, e^{M_{\infty}}\}.
\end{equation}
Hence, from \eqref{eq.18d}, \eqref{eq.9abc}, and the fact that $X(t) \in \mathcal{M}^{2}[0,T]$,  it follows that
\begin{equation}\label{eq.18bbc}
\mathbb{E}\left[\int_{0}^{t}|X(s)|^{2q(X(s))}ds\right]\leq 2(t+1)K^2\mathbb{E}\left[\int_{0}^{t}X(s)^2ds \right]<\infty.
\end{equation}
Additionally, from \eqref{eq.8aa} and \eqref{eq.9cbd}, we can obtain
\begin{equation}\label{eq.18bab}
\mathbb{E}\left[\sup_{0\leq t \leq T}|X(t)|^{2q(X(t))}\right]\leq 2K^2 \left(1+\mathbb{E}\left[\sup_{0 \leq t \leq T}|X(t)|^{2}\right]\right) <\infty.
\end{equation}
Thus, using the Fubini theorem and the fact that $g \in C_{0}^{2,1}$, we have
\begin{equation} \label{eq.19c}
\int_0^t \mathbb{E}\left[\left( X(s)^{q(X(s))}\frac{d g(X(s))}{d x}\right)^2\right] ds< \infty.
\end{equation}
Now, integrating from $0$ to $t$  in \eqref{eq.18f}, taking expectations and using the zero mean property (of the It\^{o} integral) yields
\begin{align} \label{eq.20c}
& \mathbb{E}_{x}\left[g(X(t)-g(x)\right]\nonumber\\
&=\mathbb{E}_{x} \left[ \int_0^{t} \left(\mu X(s)^{p(X(s))}\frac{dg(X(s))}{d x}+\frac{1}{2}\sigma^2 X(s)^{2q(X(s))}\frac{d^2 g(X(s))}{d x^2} \right) ds:=I(t) \right].
\end{align}
Then
\begin{align} \label{eq.21c}
|I(t)|\leq \mu \left|\int_0^{t} X(s)^{p(X(s))}\frac{dg(X(s))}{d x}ds\right|+\frac{1}{2}\sigma^2 \left|\int_0^{t} X(s)^{2q(X(s))}\frac{d^2 g(X(s))}{d x^2}ds\right|.
\end{align}
Based on the results \eqref{eq.18bbc}-\eqref{eq.19c}, $\mathbb{E}_{x}[|I(t)|]<\infty$. Note that we also have $I(t) \to I(s)$ a.s.  Thus, dividing \eqref{eq.20c} by $t$ and then letting $t \to 0^+$, applying the dominated convergence theorem (DCT), and finally using the Mean value theorem (MVT) gives
\begin{align} \label{eq.21de}
& \lim_{t \to 0^+}\frac{\mathbb{E}_{x}[g(X(t))]-g(x)}{t}\nonumber\\
& =\mathbb{E}_{x} \left[\lim_{t \to 0^+}\frac{1}{t} I(t) \right]=\mathbb{E}_{x} \left[\mu X(0)^{p(X(0))}\frac{dg(X(0))}{d x}+\frac{1}{2}\sigma^2 X(s)^{2q(X(0))}\frac{d^2 g(X(0))}{d x^2}\right]\nonumber\\
&=\frac{1}{2}\sigma^2 x^{2q(x)}\frac{d^2 g(x)}{d x^2}+\mu x^{p(x)}\frac{dg(x)}{d x}.
\end{align}
\end{proof}
\begin{remark}\label{Rem:2}
Although $\mathcal{L}$ is a degenerate elliptic operator (i.e. it fails $(\mathbf{a_3})$), the process $X(t)$ of \eqref{eq.1a} remains in $(0,\infty)$ for all $t \in [0,T]$ a.s. (Lemma \ref{Lem:2.3}). This ensures that the degeneracy at $x=0$ does not affect the well-posedness of \eqref{eq.1}.
\end{remark}
In the sequel, we use the cylinders $Q_{T}:=(0,\infty) \times [0,T)$ and $\overline{Q}_{T}:=(0,\infty) \times [0,T]$.

\section{Stochastic representation of the solution} \label{Sec3}
In this section, we will express the solution of \eqref{eq.1} as an expectation over Brownian paths starting at $x$. To do so, we use the Feynman-Kac formula (see, e.g., \cite{karatzasshreve91}).\\

\begin{theorem}\label{Thrm:3.1}
The Cauchy problem \eqref{eq.1} admits a unique bounded viscosity solution $u \in C_b(\overline{Q}_{T}) \cap W_{m,loc}^{2,1}(\overline{Q}_{T})$, $1<m<\infty$,  given by the stochastic representation
\begin{equation} \label{eq.5b}
u(x,t) = \mathbb{E}_{x}\left[e^{-\int_{0}^{t}V(X(s))ds}f(X(t)) \right]
\end{equation}
with $\|u\|_{C_b}\leq \|f\|_{C_b}$ where $X(t)$ is the solution of \eqref{eq.1a}.
\end{theorem}

\begin{proof}
Let $u(x,t)$ be the solution of \eqref{eq.1} with noting that the existence of such solution defined by the Feynman-Kac formula \eqref{eq.5b} is guaranteed by Lemma \ref{Lem:2.4}. We aim to prove that this candidate function $u(x,t)$ is indeed a viscosity solution to \eqref{eq.1}. We split the proof up into five steps.\\
\textbf{Step 1. Boundedness of $u(x,t)$.\\}
Considering the assumptions for $V$ and $f$, and applying the Jensen's inequality we have
\begin{equation} \label{eq.6}
|u(x,t)| = \left| \mathbb{E}_{x}\left[e^{-\int_{0}^{t}V(X(s))ds}f(X(t)) \right]\right|\leq \mathbb{E}_{x}\left[e^{-\int_{0}^{t}V(X(s))ds}\left| f(X(t))\right| \right] \leq \|f\|_{C_b}
\end{equation}
uniformly in  $\overline{Q}_{T}$.\\
\textbf{Step 2. The stability of $X(t)$ and continuity of $u(x,t)$}.\\
Fix $t \in [0,T]$ and consider the sequences $x_n \to x$ as $n \to \infty $ with $x_n, x \in (0,\infty)$. Let $X_{x_n}(t)$ and $X_{x}(t)$ denote the solution (of \eqref{eq.1a}) starting from positions $x_n$ and $x$, respectively. Then for any integer $m\geq 2$ we have
\begin{align} \label{eq.7}
\sup_{s \leq t}|X_{x_n}(s)-X_{x}(s)|^{m} & \leq 3^{m-1} |x_n-x|^m \nonumber \\
& + 3^{m-1}\mu^m\sup_{s \leq t}\left|\int_0^s \left(X_{x_n}(r)^{p(X_{x_n}(r))}-X_{x}(r)^{p(X_{x}(r))}\right) dr\right|^{m} \nonumber \\
& + 3^{m-1} \sigma^m \sup_{s \leq t}\left|\int_0^s \left(X_{x_n}(r)^{q(X_{x_n}(r))}-X_{x}(r)^{q(X_{x}(r))}\right) dW(r)\right|^{m}.
\end{align}
Applying the Burkholder-Davis-Gundy (BDG) inequality and the H\"{o}lder inequality  together yields
\begin{align} \label{eq.8}
\mathbb{E}\left[\sup_{s \leq t}|X_{x_n}(s)-X_{x}(s)|^m\right] & \leq 3^{m-1} |x_n-x|^m \nonumber \\
& + 3^{m-1}\mu^m L^m t^{m-1} \int_0^t \mathbb{E}\left[\sup_{s \leq r}\left|X_{x_n}(r)-X_{x}(r)\right|^m\right] dr \nonumber \\
& + 3^{m-1} C_m\sigma^m L^m t^{(m-2)/2} \int_0^t \mathbb{E}\left[\sup_{s \leq r}\left|X_{x_n}(r)-X_{x}(r)\right|^m\right] dr \nonumber \\
& \leq 3^{m-1} |x_n-x|^m + \hat{L}\int_0^t \mathbb{E}\left[\sup_{s \leq r}\left|X_{x_n}(r)-X_{x}(r)\right|^m\right] dr.
\end{align}
where $\hat{L}=3^{m-1}L^m(\mu^m t^{m-1}+ C_m\sigma^m t^{(m-2)/2})$. Now, applying the Gronwall inequality yields
\begin{align} \label{eq.9}
\mathbb{E}\left[\sup_{s \leq t}|X_{x_n}(s)-X_{x}(s)|^m\right] & \leq 3^{m-1}e^{T\hat{L}} |x_n-x|^m ,
\end{align}
which means that $\sup_{s \leq t}|X_{x_n}(s)-X_{x}(s)|\to 0$ in $L^{m}$, and hence in probability, when $x_n \to x$.
Since we don't have a closed form for the solution $X(t)$ (of \eqref{eq.1a}), we will use joint continuity of the stochastic flow argument. Under the Lipschitz and linear growth conditions on the coefficients $\mu x^{p(x)}$ and $\sigma x^{q(x)}$ of \eqref{eq.1a} and the fact that the solution $X(t)$ is non-explosive in $(0,\infty)$, there exists a version of the solution $X(t)$ such that, for almost every sample path $\omega \in \Omega$, the map
\begin{align} \label{eq.10}
(x,t) \mapsto X_{x}(t;\omega)
\end{align}
is jointly continuous in $(x,t)$ a.s. \cite{kunita97}. Now, fix sample path $\omega \in \Omega$, and define the random function
\begin{align} \label{eq.11}
\Phi_{f}(x,t;\omega)= e^{-\int_{0}^{t}V(X_{x}(s;\omega))ds}f(X_{x}(t;\omega)).
\end{align}
Thus, considering the continuity of the maps
\begin{equation} \label{eq.12}
(x,t) \mapsto f(X_{x}(t;\omega)), \quad (x,t) \mapsto \int_{0}^{t}V(X_{x}(s;\omega))ds
\end{equation}
we concluded that for almost every sample path $\omega \in \Omega$, the function $(x,t) \mapsto \Phi_{f}(x,t;\omega)$ is continuous. On the other hand, since there exists $M>0$ such that $|f(x)|\leq M$ and $V(x)\geq 0$ for all $x \in (0, \infty)$, we have
\begin{align} \label{eq.13}
|\Phi_{f}(x,t;\omega)|\leq M \quad \text{uniformly in } (x,t;\omega).
\end{align}
Now, take any sequence $(x_n,t_n) \in \overline{Q}_{T}$ such that  $(x_n,t_n) \to (x,t) \in \overline{Q}_{T}$. By the continuity of $\Phi_{f}$ in $(x,t)$ we have $\Phi_{f}(x_n,t_n;\omega) \mapsto \Phi_{f}(x,t;\omega)$ for almost every sample path $\omega \in \Omega$. In conclusion, using the DCT, it follows
\begin{align} \label{eq.14}
\lim_{n \to \infty}u(x_n,t_n)=\lim_{n \to \infty} \mathbb{E}\left[\Phi_{f}(x_n,t_n;\omega)\right]=\mathbb{E}\left[\Phi_{f}(x,t;\omega)\right]=u(x,t).
\end{align}
\textbf{Step 3. Viscosity property of $u(x,t)$}.\\
Since the solution process $X(t)$ (of \eqref{eq.1a}) has the strong Markov property, $u(x,t)$ is bounded-continuous and $V(x)$ is nonnegative-continuous, we can use (DPP) to show $u(x,t)$ is both a viscosity sub- and supersolution.\\
We first show that $u(x,t)$ is a viscosity subsolution. Fix a point $(x_0,t_0) \in Q_T$ and let $\varphi \in C_{0}^{2,1}$ be a test function such that $u-\varphi$ has a local maximum at $(x_0,t_0)$. Note that, since $\varphi$ is a test function, we can replace $\varphi$ by $\varphi+C$ such that $u(x_0,t_0)=\varphi(x_0,t_0)$ which implies that $u(x,t) \leq \varphi(x,t)$ for all points $(x,t)$ near $(x_0,t_0)$.
More specifically, there exists $\rho>0$ such that on the open neighborhood
\begin{equation} \label{eq.15a}
U_{\rho}=(x_0-\rho,x_0+\rho) \times (t_0-\rho,t_0+\rho)
\end{equation}
we have
\begin{equation} \label{eq.16}
u(x,t)-\varphi(x,t)\leq u(x_0,t_0)-\varphi(x_0,t_0),\quad \forall (x,t) \in U_{\rho}.
\end{equation}
To make sure the process stays inside $U_{\rho}$, we define the stopping time
\begin{equation} \label{eq.16a}
\tau_\rho= h \wedge \inf\{s\in [0,h]:\, (X(s), t_0-s) \notin U_{\rho}\},\quad 0<h\leq \rho \leq t_0
\end{equation}
such that for every sample path $\omega \in \Omega$ we have
\begin{equation} \label{eq.16b}
u(X(\tau_\rho), t_0-\tau_\rho)-\varphi(X(\tau_\rho), t_0-\tau_\rho)\leq u(x_0,t_0)-\varphi(x_0,t_0)\quad \text{a.s.}
\end{equation}
Now, we argue by contradiction and assume that $u(x,t)$ is not a viscosity subsolution; that is,
\begin{equation} \label{eq.14a}
\partial_{t}\varphi(x_0,t_0)-\mathcal{L}\varphi(x_0,t_0)+V(x_0)\varphi(x_0,t_0)>0.
\end{equation}
Then applyin the DPP with stopping times gives
\begin{equation} \label{eq.15}
u(x_0,t_0) = \mathbb{E}_{x_0}\left[e^{-\int_{0}^{\tau_\rho}V(X(s))ds}u(X(\tau_\rho),t_0-\tau_\rho) \right]
\end{equation}
from which one can obtain
\begin{equation} \label{eq.17}
0\leq \mathbb{E}_{x_0}\left[e^{-\int_{0}^{\tau_\rho}V(X(s))ds}\varphi(X(\tau_\rho),t_0-\tau_\rho)-\varphi(x_0,t_0) \right].
\end{equation}
Define the process $Y(s):=\varphi(X(s),t_0-s)$ for $s \in [0,\tau_\rho]$. If we apply It\^{o}'s formula it reads
\begin{align} \label{eq.18}
dY(s)&=\left[-\partial_{t}\varphi+\mu X(s)^{p(X(s))}\partial_{x}\varphi +\frac{1}{2}\sigma^2 X(s)^{2q(X(s))}\partial^{2}_{x}\varphi \right](X(s),t_0-s)ds\nonumber\\
&+\frac{1}{2}\sigma^2 X(s)^{2q(X(s))}\partial_{x}\varphi(X(s),t_0-s) dW(s).
\end{align}
Integrating from $0$ to $\tau_\rho$, taking the expectation and considering the zero mean property gives
\begin{align} \label{eq.20}
& \mathbb{E}_{x_0}\left[\varphi(X(\tau_\rho), t_0-\tau_\rho) - \varphi(x_0, t_0)\right]\nonumber\\
&=\mathbb{E}_{x_0} \left[ \int_0^{\tau_\rho} \left( -\partial_{t}\varphi+\mu X(s)^{p(X(s))}\partial_{x}\varphi +\frac{1}{2}\sigma^2 X(s)^{2q(X(s))}\partial^{2}_{x}\varphi\right)(X(s), t_0-s) ds \right].
\end{align}
Put $z(\tau_\rho)=\int_{0}^{\tau_\rho}V(X(s))ds$ and $K(\tau_\rho)=e^{-z(\tau_\rho)}$. Applying the Taylor approximation to $K(\tau_{\rho})$ around zero gives
\begin{align} \label{eq.21}
e^{-\int_0^{\tau_{\rho}} V(X(s)) ds} \approx 1 - \tau_{\rho} V(x_0)+o(\tau_{\rho}),
\end{align}
where $o(\tau_{\rho})$ corresponds to terms of order 2 and higher. Plugging these information into \eqref{eq.17} and using \eqref{eq.20} gives
\begin{align} \label{eq.22}
0 & \leq \mathbb{E}_{x_0}\left[\varphi(X(\tau_{\rho}),t_0-\tau_{\rho})-\varphi(x_0,t_0) \right]-\mathbb{E}_{x_0}\left[\tau_{\rho}  c(x_0)\varphi(X(\tau_{\rho}),t_0-\tau_{\rho})\right]\nonumber\\
&+\mathbb{E}_{x_0}\left[o(\tau_{\rho})\varphi(X(\tau_{\rho}),t_0-\tau_{\rho})\right]\nonumber\\
& = \mathbb{E}_{x_0} \left[ \int_0^{\tau_{\rho}} \left( -\partial_{t}\varphi+\mu X(s)^{p(X(s))}\partial_{x}\varphi +\frac{1}{2}\sigma^2 X(s)^{2q(X(s))}\partial^{2}_{x}\varphi\right)(X(s), t_0-s) ds \right]\nonumber\\
& - \mathbb{E}_{x_0}\left[\tau_{\rho}  V(x_0)\varphi(X(\tau_{\rho}),t_0-\tau_{\rho})\right]+\mathbb{E}_{x_0}\left[o(\tau_{\rho})\varphi(X(\tau_{\rho}),t_0-\tau_{\rho})\right].
\end{align}
Note that the integrand and all other terms are continuous and bounded on the small time interval $(0,\tau_{\rho})$ and $\tau_{\rho}$ is a stopping time. Thus, dividing \eqref{eq.22} by $\tau_{\rho}$ and then letting $\rho \to 0^+$ (and hence $\tau_{\rho} \to 0^+$) as well as applying the DCT along with the MVT gives
\begin{align} \label{eq.23}
0 &  \leq \mathbb{E}_{x_0} \left[\lim_{\tau_{\rho} \to 0^+}\frac{1}{\tau_{\rho}} \int_0^{\tau_{\rho}} \left( -\partial_{t}\varphi+\mu X(s)^{p(X(s))}\partial_{x}\varphi +\frac{1}{2}\sigma^2 X(s)^{2q(X(s))}\partial^{2}_{x}\varphi\right)(X(s), t_0-s) ds \right]\nonumber\\
& - \mathbb{E}_{x_0} \left[\lim_{\tau_{\rho} \to 0^+}\left(V(x_0)\varphi(X(\tau_{\rho}),t_0-\tau_{\rho})\right)\right]+\mathbb{E}_{x_0}\left[\lim_{\tau_{\rho} \to 0^+} \frac{o(\tau_{\rho})}{\tau_{\rho}} \varphi(X(\tau_{\rho}),t_0-\tau_{\rho})\right]
\end{align}
which implies
\begin{equation} \label{eq.24}
\partial_{t}\varphi(x_0,t_0)-\mathcal{L}\varphi(x_0,t_0)+V(x_0)\varphi(x_0,t_0) \leq 0.
\end{equation}
However, this contradicts \eqref{eq.14a}. Thus, $u(x,t)$ must be a subsolution. Applying a similar argument at a local minimum point confirms that $u(x,t)$ is also a viscosity supersolution. Thus $u(x,t)$ is a bounded viscosity solution to \eqref{eq.1}.\\
\textbf{Step 4. Uniqueness of $u(x,t)$}.\\
We will use the comparison principle for viscosity solutions to show  the uniqueness of $u(x,t)$  \cite{cil92}. First note that \eqref{eq.1} is continuous in all its arguments, proper ($V(x)\geq 0$), has a bounded viscosity solution, and include a degenerate elliptic operator $\mathcal{L}$, which makes it a degenerate parabolic equation. Therefore, the comparison principle applies on $\mathbb{R}$. To do so, we transform \eqref{eq.1} to a uniformly parabolic equation via change of variables. More specifically, to handle the variable exponent diffusion coefficient $\sigma^2 x^{2q(x)}$ effectively, we transform to logarithmic coordinates
\begin{equation}\label{eq.27a}
y = \log(x), \quad v(y,t) := u(e^y, t).
\end{equation}
Under this log-transform, $x \mapsto y(x)=\log(x)$, \eqref{eq.1} turns into the uniformly parabolic equation
\begin{equation}\label{eq.27}
\begin{cases}
\begin{array}{rll}
\partial_t v &= A(y)\partial^2_y v+ B(y)\partial_y v - C(y)v,\,\, (y,t) \in \mathbb{R} \times (0,T),\\
v(y,0)&=f(e^y),\, y \in \mathbb{R},  \tag{{${\mathcal{{P}}_y}$}}
\end{array}
\end{cases}
\end{equation}
where
\begin{align}
A(y) =  \frac{1}{2}\sigma^2 e^{y(2q(e^y)-2)}, \quad  B(y)  = \mu e^{y(p(e^y)-1)}-A(y), \quad C(y)=V(e^y).
\end{align}
Since the log-transform is a $C^{\infty}-$ diffeomorphism between $(0,\infty)$ and $\mathbb{R}$; that is, it preserves the regularity back and forth. Thus, if $u(x,t)$ is a unique bounded viscosity solution of \eqref{eq.1}, then $v(y,t)$ is a unique bounded viscosity solution of \eqref{eq.27}, and vice versa.
Assume that $u_1, u_2$ are two different bounded viscosity solutions to \eqref{eq.1} with the same initial data, i.e.  $u_{i}(x,0)=f(x)$, $i=1,2$. Then the transformation gives $u_{i}(e^y,t)=v_{i}(y,t)$ and $v_{i}(y,0)=f(e^y)$. Since $v_{1}$ is a bounded viscosity solution, it is a bounded viscosity subsolution; and since $v_{2}$ is a bounded viscosity solution, it is a bounded viscosity supersolution. Then by the comparison principle, the solutions remain ordered for all time; that is, $v_{1}(y,t)\leq v_{2}(y,t)$ for all $(y,t) \in \mathbb{R} \times [0,T]$. If we swap the roles and proceed in the same fashion, we obtain $v_{2}(y,t)\leq v_{1}(y,t)$ for all $(y,t) \in \mathbb{R} \times [0,T]$. Therefore, $v_{1}(y,t)= v_{2}(y,t)$ for all $(y,t) \in \mathbb{R} \times [0,T]$. Finally, by transforming $v_{1}, v_{2}$ back, we conclude that $u_1(x,t)=u_2(x,t)$ for all $(x,t) \in \overline{Q}_{T}$.\\
\textbf{Step 5. Local Sobolev regularity.}\\
Let's rewrite \eqref{eq.1} as
\begin{equation} \label{eq.28}
\partial_{t}u(x,t)-\mathcal{L}u(x,t)=-V(x)u(x,t).
\end{equation}
Since $\gamma(x,t):=V(x)u(x,t)$ is bounded and continuous on a compact cylinder $\mathcal{Q} \subset (0,\infty) \times [0,T]$, $\gamma \in L^m(\mathcal{Q})$ for $m\geq 1$. The coefficients of the operator $\mathcal{L}$ are Lipschitz continuous and have linear growth, and hence, they're continuous and bounded on $\mathcal{Q}$. Moreover, the operator $\mathcal{L}$ is uniformly elliptic on $\mathcal{Q}$. In conclusion, by the regularity theory for parabolic equations (see e.g., \cite{lsu68}), the solution $u(x,t)$ of \eqref{eq.1} belongs to $ W_{m}^{2,1}(\mathcal{Q})$ with an a-priori estimate
\begin{equation} \label{eq.29}
\|u\|_{W_{m}^{2,1}(\mathcal{Q})}=\|u\|_{L^{m}(\mathcal{Q})}+\|\partial_tu\|_{L^{m}(\mathcal{Q})}+\|\partial_xu\|_{L^{m}(\mathcal{Q})}+\|\partial^2_xu\|_{L^{m}(\mathcal{Q})}<\infty.
\end{equation}
\end{proof}

\section{Numerical implementation}
In this section, we present a numerical validation of the theoretical link between a one-dimensional parabolic Cauchy problem \eqref{eq.1} and its corresponding stochastic differential equation \eqref{eq.1a}. We numerically solve \eqref{eq.1}  using a Crank--Nicolson finite difference scheme on a transformed log-grid. Concurrently, we simulate \eqref{eq.1a}  using an Euler-Maruyama scheme and compute the expectation from the Feynman--Kac formula via a Monte Carlo method. We compare the results for three cases of state-dependent exponents and find strong consistency, validating the stochastic representation \eqref{eq.5b}.

\subsection{Parameters}
The following single set of parameters are used for all three test cases to ensure a fair comparison as well as to isolate the impact of the variable exponent functions.
\begin{table}[h!]
\centering
 \begin{tabular}{||c c||}
 \hline
Parameter & Value \\ [0.5ex]
 \hline\hline
Drift constant & $\mu = 0.1$ \\
Diffusion constant & $\sigma = 0.2$ \\
Potential function & $V(x) = 0.1$  \\
Initial condition & $f(x) = e^{-0.1x}$ \\
Terminal time & $T = 1.0$ \\
Truncation domain & $x \in [0.1, 50.0]$ \\ [1ex]
 \hline
 \end{tabular}
 \caption{Physical parameters.}
\end{table}

\begin{table}[h!]
\centering
 \begin{tabular}{||c c c||}
 \hline
Parameter & PDE & SDE  \\ [0.5ex]
 \hline\hline
Spatial points & $N_x = 400$ & $N_x = 400$  \\
Time steps & $N_t = 400$ & $N_{\text{steps}} = 400$ \\
MC paths & --- & $N_{\text{paths}} = 20{,}000$ \\
Main variates & --- & 10,000 \\
Antithetic variates & --- & 10,000 \\
Spatial step & $\Delta y \approx 0.0156$ & --- \\
Time step & $\Delta t = 0.0025$ & $\Delta t_{\text{SDE}} = 0.0025$ \\ [1ex]
 \hline
 \end{tabular}
 \caption{Discretization parameters.}
\end{table}

\subsection{PDE Solver: Crank-Nicolson Scheme on a Log-Grid}
Solving \eqref{eq.1} directly on a uniform grid for $x$ is inefficient due to the semi-infinite domain and the non-linear nature of the coefficients. We employ the standard log-transform as we did in \eqref{eq.27a}, and use \eqref{eq.27} as the transformed PDE version of \eqref{eq.1} to use in the numerical experiments. The transformed PDE \eqref{eq.27} is solved using the \emph{Crank--Nicolson} scheme, which is unconditionally stable and second-order accurate in both time ($\Delta t$) and space ($\Delta y$). This produces a tridiagonal system at every time step.
Since the log-transform maps $(0,\infty)$ to $\mathbb{R}$, the computational (truncated) domain becomes $[y_{\min}, y_{\max}] = [\log(r), \log(R)]$. Thus, by homogeneous Neumann boundary conditions (BCs) $\partial_{x}u=0$  at the truncation boundaries $x=r, R$ are transformed into $\partial_{y}v=0$ at $y=\log(r), \log(R)$.
These boundary conditions are implemented using mirrored ghost points to maintain the second-order accuracy of the
finite difference scheme:
\begin{equation}
v_{-1} = v_1 \quad \text{and} \quad  v_{N_x} = v_{N_x-2}, \quad i = 0,1,\ldots,N_x-1.
\end{equation}
\begin{remark}\label{Rem:3}
Homogeneous Neumann BCs are the natural choice for this type of problems because the original problem on $(0,\infty)$ has no boundary conditions, so the truncation to $[r,R]$ requires conditions that don't impose artificial constraints, and they are fundamentally consistent with the Feynman-Kac representation, which derives the solution from the SDE without any boundary data, ensuring that the PDE and probabilistic methods yield the same result.
\end{remark}

\subsection{SDE Solver: Euler-Maruyama Monte Carlo}
\begin{enumerate}
    \item \textbf{Time discretization.} We simulate paths of \eqref{eq.1a} with the Euler-Maruyama scheme. A single path is generated by the iterative relation:
    $$
      X_{n+1} = X_n + \mu X_n^{p(X_n)} \Delta t + \sigma X_n^{q(X_n)} \sqrt{\Delta t}\, Z_n,
    $$
    where $Z_n \sim \mathcal{N}(0,1)$.

    \item \textbf{Monte Carlo estimation.} The expectation in \eqref{eq.5b} is approximated by averaging over $N_{\text{paths}}$ simulated paths. For each starting point $x$, the solution is estimated as:
   $$
      u(x,T) \approx \frac{1}{N_{\text{paths}}} \sum_{j=1}^{N_{\text{paths}}}
      \biggl[
        f\bigl(X_T^{(j)}\bigr)
        e^{\left(-\sum_{n=0}^{N_t-1} V\bigl(X_{t_n}^{(j)}\bigr)\,\Delta t \right)}
      \biggr].
  $$

    \item \textbf{Variance reduction.} To improve the stability and accuracy of the Monte Carlo estimate, antithetic variates and common random numbers techniques are used.
\end{enumerate}

\subsection{Analysis}

We compare three exponent configurations. We use the same procedure for all three cases: Solve the PDE with Crank-Nicolson using $p_i(x),q_i(x)$, simulate the SDE with Euler-Maruyama, and estimate the Feynman-Kac expectation with $N_{\text{paths}}=20,000$.\\

\noindent\textbf{Case 1:}  $p_1(x) = 1 + \dfrac{0.30}{1+x^{1.2}}$, \quad $q_1(x) = 1 + \dfrac{0.40}{1+x^{2}}$ (polynomial-decay).\\

\begin{figure}[H]
  \centering
  \includegraphics[width=0.95\textwidth]{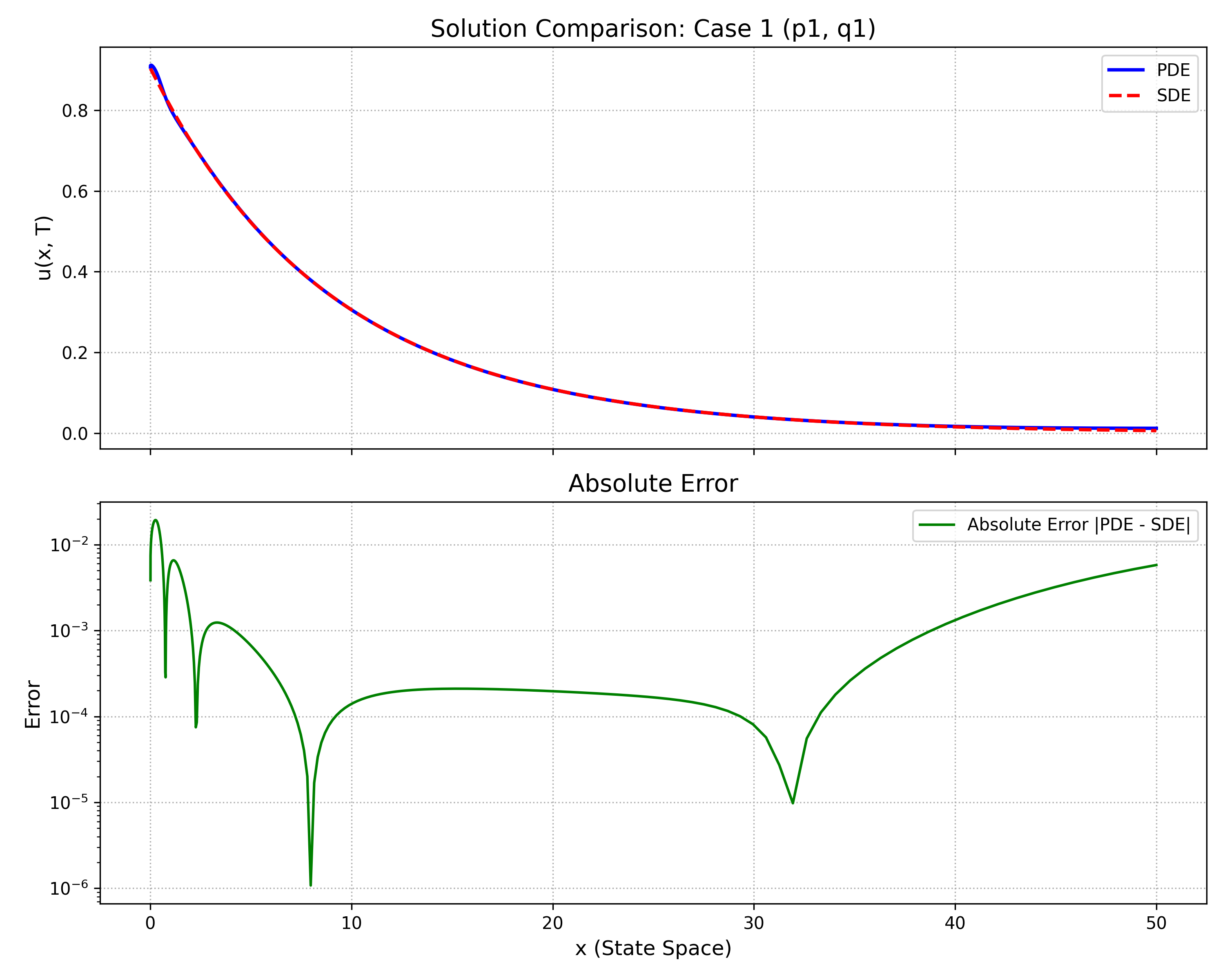}
  \caption{Comparison of PDE and SDE solutions for variable exponents $p_1(x)$ and $q_1(x)$.}
\end{figure}

The two solutions are plotted in Figure 1 (Top). The solid blue line (PDE solution) and the dashed red line (SDE Monte Carlo solution) are visually indistinguishable.
The bottom panel of Figure 1 shows the absolute error $|u_{\text{PDE}}-u_{\text{SDE}}|$. The error is on the order of $\approx 10^{-4}$, which is exceptionally low. This residual error is a combination of the PDE’s discretization error and the SDE’s time-discretization and statistical errors. The spiky nature of the error plot is characteristic of the underlying statistical noise from the Monte Carlo simulation. This result provides a strong numerical validation of the Feynman- Kac formula for this set of exponents.\\

\noindent\textbf{Case 2:} $p_2(x) = 1 + 0.2 e^{-x}$, \quad $q_2(x) = 1 + 0.3 e^{-x}$ (exponential-decay).\\

\begin{figure}[H]
  \centering
  \includegraphics[width=0.95\textwidth]{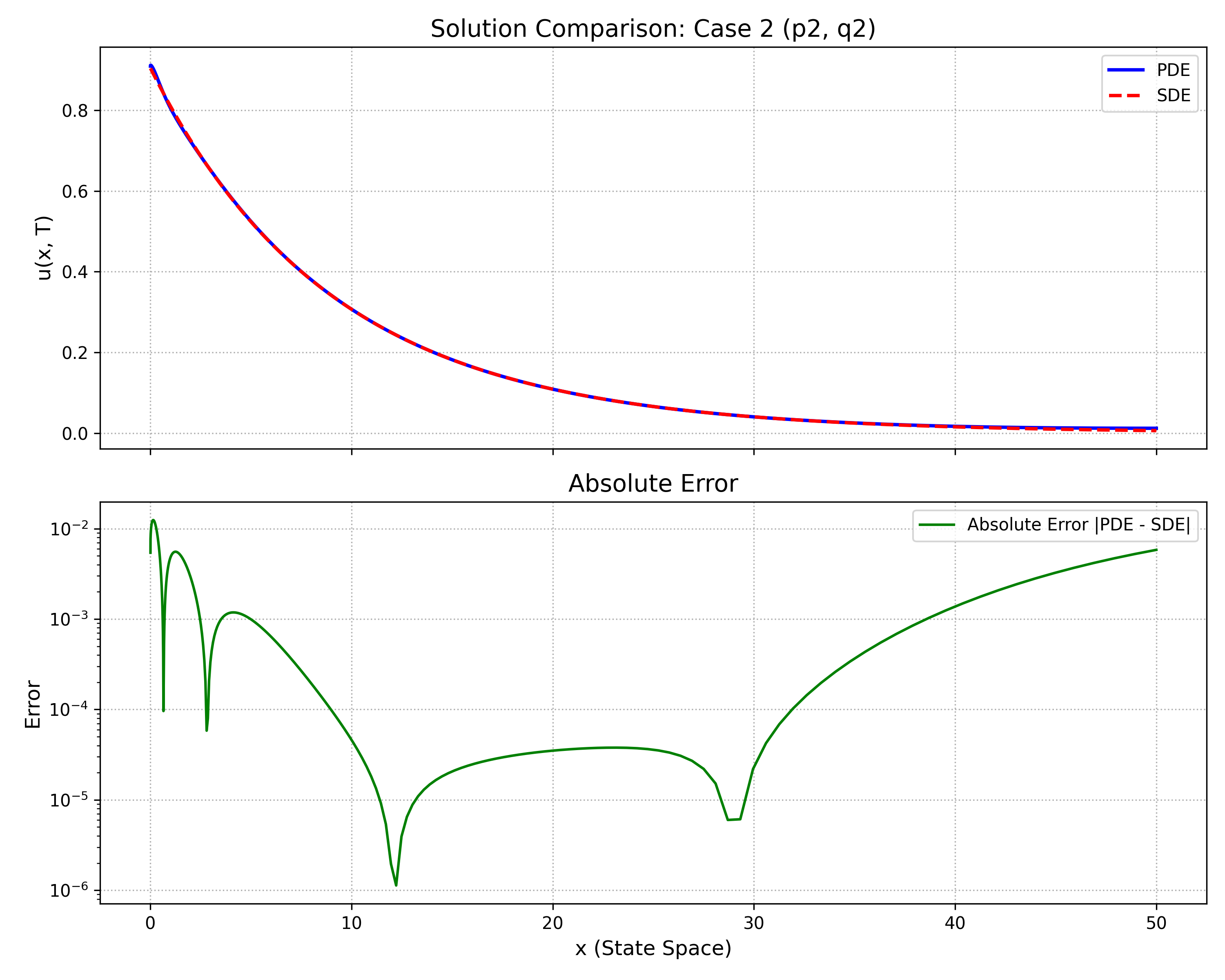}
  \caption{Comparison of PDE and SDE solutions for variable exponents $p_2(x)$ and $q_2(x)$.}
\end{figure}

As shown in Figure 2, the results are qualitatively identical to Case 1. The PDE and SDE solutions are in excellent agreement. The error, plotted in the bottom panel of Figure 2, is again on the order of $\approx 10^{-4}$. This demonstrates that the theoretical validation is robust and not dependent on a specific functional form for the variable exponents. The methods are accurate for both polynomial-decay
and exponential-decay exponents.\\

\noindent\textbf{Case 3:} $p_3(x) = q_3(x) = 1$ (GBM base model).\\

\begin{figure}[H]
 \centering
  \includegraphics[width=0.95\textwidth]{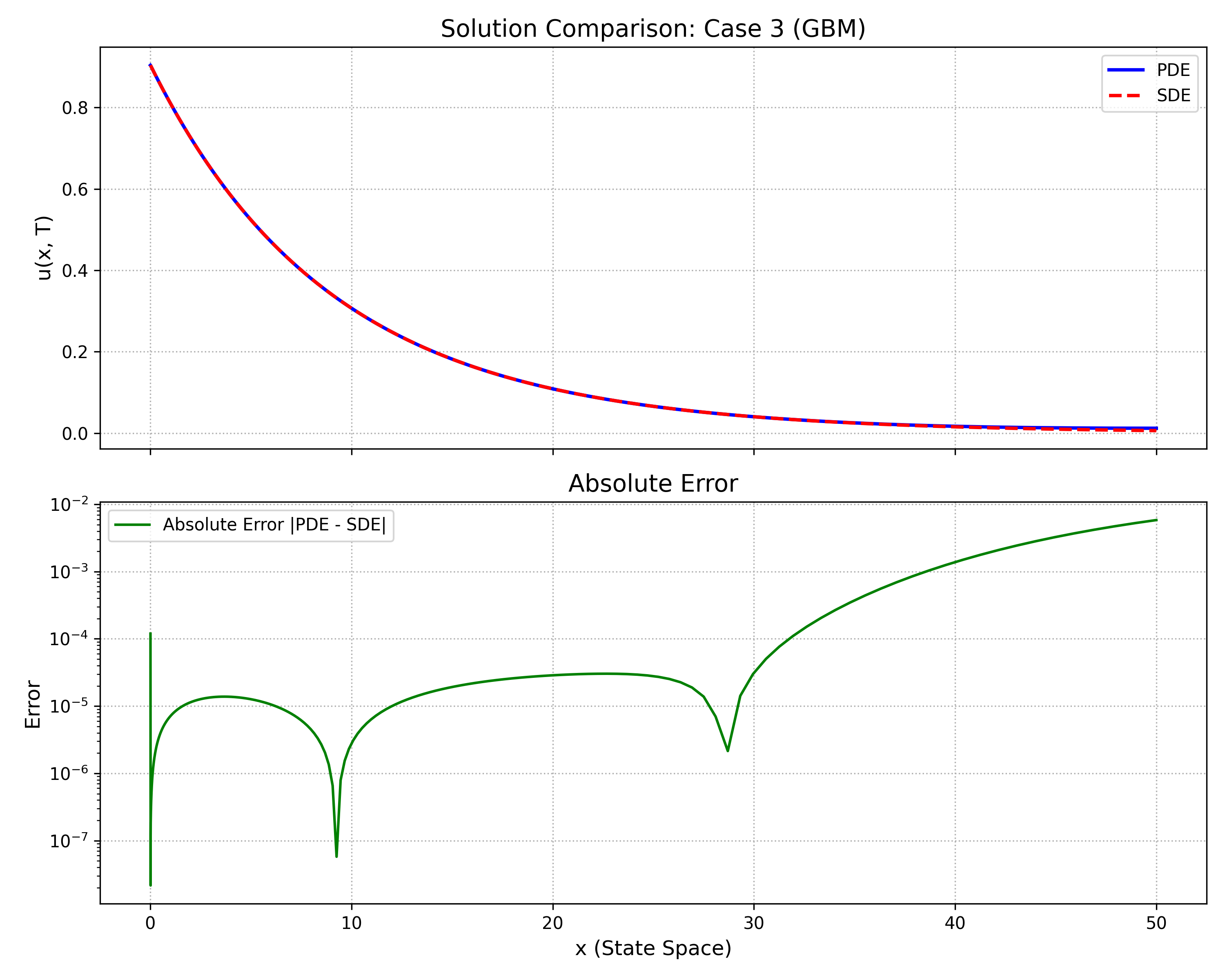}
  \caption{Comparison of PDE and SDE solutions for the baseline case $p_3(x)=q_3(x)=1$.}
\end{figure}

Figure 3 shows that the agreement between the PDE and SDE solutions remains consistent. Note that this case serves as a crucial control for our numerical methods.
The consistently low error magnitude $(\approx 10^{-4})$, matching that of the variable-exponent cases, strongly suggests that both the Crank-Nicolson and Euler-Maruyama solvers are correctly implemented. The errors observed in Cases 1 and 2 are not due to the variable-exponents, but rather reflect the intrinsic numerical accuracy of the methods. This reinforces the reliability of the overall computational framework.

\subsection{Conclusion}
This numerical experiment provides strong empirical evidence for Theorem \ref{Thrm:3.1}, which states that the bounded viscosity solution to \eqref{eq.1} is uniquely given by the Feynman-Kac representation \eqref{eq.5b}.
Across all three test cases, the solutions produced by the Crank-Nicolson PDE solver and those estimated via Monte Carlo SDE simulation show strong consistency, with absolute errors on the order of $\approx 10^{-4}$. This results suggest strong numerical support to the underlying theoretical framework.

\section{The contributions of the paper}

In this paper, we study a new class of parabolic Cauchy problem with variable exponent coefficients, and establish existence, uniqueness, and stochastic representation of viscosity solutions. Our main contributions are:
\begin{itemize}
 \item \textit{Generalization}: The proposed model, \eqref{eq.1a}, generalizes many well-know models. For example, if one let $p(x)=q(x)=1$, \eqref{eq.1a} becomes the GBM.
          Note that a key feature of ($\mathbf{h_2}$) is that $p(x) \to 1$ and $q(x) \to 1$ as $x \to \infty$. This implies our generalized model \eqref{eq.1a} behaves like the GBM for large $x$, but differently for small $x$. The proposed model also generalizes the CEV model if one let $p(x)=1$ and $ q(x)\geq 1$ (excluding the case $0\leq q(x)<1$).

    \item \textit{Well-posedness}: We prove that under suitable hypotheses on the variable exponents $p(\cdot), q(\cdot)$, including boundedness, differentiability, asymptotic behavior, and controlled derivatives, the Cauchy problem \eqref{eq.1} admits a unique bounded viscosity solution.

    \item \textit{SDE analysis}: We establish the existence, uniqueness, and positivity of strong solutions to the associated SDE \eqref{eq.1a} with state-dependent variable exponents. Crucially, we prove that solutions remain strictly positive and satisfy moment bounds, overcoming the technical challenges posed by the nonlinear coefficients.

    \item \textit{Feynman-Kac representation}: We demonstrate that the viscosity solution of \eqref{eq.1} admits the stochastic representation
   $$
    u(x,t) = \mathbb{E}_x\left[e^{-\int_0^t V(X(s))ds}f(X(t))\right],
    $$
    where $X(t)$ solves \eqref{eq.1a}. This extends the classical Feynman-Kac formula to the variable exponent setting.

    \item \textit{Numerical validation}: We design and implement comprehensive numerical experiments comparing finite difference solutions (Crank-Nicolson on logarithmic grids) with Monte Carlo simulations. Our experiments validate the theoretical results and provide quantitative error estimates across multiple variable exponent configurations.
\end{itemize}

\section*{Data usage statement}
\noindent All code used in the numerical experiments is original and requires no external data sources beyond standard scientific computing libraries of Python programming language. All figures are obtained purely from simulating the equations studied in the paper.

\section*{Acknowledgments}
\noindent This work was supported by Athabasca University Research Incentive Account [140111 RIA].
\section*{ORCID}
\noindent https://orcid.org/0000-0002-6001-627X

\end{document}